\documentclass[11pt]{article}
\usepackage[utf8]{inputenc}
\usepackage[top=2.54cm, bottom=2.54cm, left=2.5cm, right=2.5cm]{geometry}
\usepackage[tbtags]{amsmath}
\usepackage{amssymb}
\usepackage{amsthm}
\usepackage{appendix}
\usepackage{latexsym}
\usepackage{mathrsfs}
\usepackage{xcolor}
\usepackage{wasysym}
\usepackage{fancyhdr}
\usepackage[colorlinks=true, linkcolor=red, citecolor=blue, urlcolor=magenta]{hyperref}
\usepackage{float}
\usepackage{graphicx}

\usepackage[numbers,sort&compress]{natbib}
\usepackage{pict2e}
\usepackage{tikz}
\usepackage{tikz-3dplot}
\usepackage{tkz-graph}
\usetikzlibrary{arrows,shapes,chains}
\usepackage{enumerate}
\usepackage{verbatim}

\allowdisplaybreaks[4]

\renewcommand{\phi}{\varphi}
\renewcommand{\epsilon}{\varepsilon}

\newtheorem{theorem}{Theorem}[section]
\newtheorem{lemma}[theorem]{Lemma}
\newtheorem{conjecture}[theorem]{Conjecture}
\newtheorem{problem}[theorem]{Problem}

\newtheorem{corollary}[theorem]{Corollary}
\theoremstyle{definition}

\newtheorem{claim}[theorem]{Claim}

\begin{document}

\title{Hypergraph Erd\H{o}s--Rogers functions with consecutive clique sizes}

\author{
Qizhong Lin\footnote{Center for Discrete Mathematics, Fuzhou University,
Fuzhou, 350108 P.~R.~China. Email: {\tt linqizhong@fzu.edu.cn}. Supported in part by National Key R\&D Program of China (Grant No. 2023YFA1010202) and NSFC (No.\ 12571361).} 
\;\; and\;\;
Lin Niu\footnote{Center for Discrete Mathematics, Fuzhou University, Fuzhou, 350108 P.~R.~China. Email: {\tt 1539166573@qq.com}.} 
}
\date{}

\maketitle

\begin{abstract}
For integers \(k\le s<t\), the hypergraph Erd\H{o}s--Rogers function
\(f^{(k)}_{s,t}(n)\) is the largest integer \(m\) such that every
\(n\)-vertex \(K_t^{(k)}\)-free \(k\)-graph contains a set of \(m\) vertices
spanning no copy of \(K_s^{(k)}\). We prove that, for every fixed \(s\ge4\),
\[
        f^{(4)}_{s,s+1}(n)=(\log n)^{o(1)},
\]
thereby resolving a problem posed by Conlon, Fox and Sudakov.
The key input is a new \(3\)-uniform estimate: 
\(f^{(3)}_{s,s+1}(n)=O(\frac{\log n}{\log\log n})\) for every fixed \(s\ge3\), which improves the
logarithmic upper bound of Dudek and Mubayi. The proof develops a probabilistic
pair-coloring construction based on a robust auxiliary palette and hypergraph
containers. As a further consequence, we obtain
\(f^{(k)}_{k+1,k+2}(n)=(\log_{(k-3)} n)^{o(1)}\) for every fixed \(k\ge5\),
making substantial progress towards a conjecture of Mubayi and Suk.
\end{abstract}

\section{Introduction}
\label{sec:introduction}

A recurring theme in Ramsey theory is to understand what structure is forced to
remain after a larger configuration has been forbidden.  The Erd\H{o}s--Rogers
function gives a natural quantitative form of this question: in a host graph or
hypergraph avoiding a prescribed clique, how large a vertex set must one find
that already avoids a smaller clique?

For integers \(k\le s<t\), let \(K_t^{(k)}\) denote the complete
\(k\)-uniform hypergraph, or \(k\)-graph, on \(t\) vertices.  We define
\(f_{s,t}^{(k)}(n)\) to be the largest integer \(m\) such that every
\(n\)-vertex \(K_t^{(k)}\)-free \(k\)-graph contains a set of \(m\) vertices
spanning no copy of \(K_s^{(k)}\).  When \(s=k\), a \(K_k^{(k)}\)-free set is
precisely an independent set.  Thus \(f_{k,t}^{(k)}(n)\ge m\) if and only if
every \(n\)-vertex \(K_t^{(k)}\)-free \(k\)-graph contains an independent set
of size \(m\), or equivalently
$
        r(K_t^{(k)},K_m^{(k)})\le n.
$
In this sense, \(f_{k,t}^{(k)}(n)\) is the inverse form of the corresponding
off-diagonal Ramsey number. Thus the Erd\H{o}s--Rogers function enters the non-Ramsey setting at \(s\ge k+1\).

For \(k=2\), the function \(f_{s,t}^{(2)}(n)\) was introduced by
Erd\H{o}s and Rogers~\cite{ER}, who asked how large an induced
\(K_s\)-free subgraph must appear in every \(n\)-vertex \(K_t\)-free graph.
The adjacent case \(t=s+1\) is the most restrictive one and has subsequently been studied
by Bollob\'as and Hind~\cite{BH}, Krivelevich~\cite{K94,K95}, Alon and
Krivelevich~\cite{AK}, Dudek and R\"odl~\cite{DR2011},
Wolfovitz~\cite{W2013}, Dudek, Retter and R\"odl~\cite{DRR},
Mubayi and Verstra\"ete~\cite{MV25}, and others. For \(s=2\), the classical results of Ajtai, Koml\'os and
Szemer\'edi~\cite{AKS} and Kim~\cite{Kim} imply that
\[
        f_{2,3}^{(2)}(n)
        =
        \Theta\bigl(\sqrt{n\log n}\bigr).
\]
Very recently, Morris, Sahasrabudhe and Verstra\"ete~\cite{MSV26}
proved an upper bound matching the lower bound that follows from a theorem
of Joret, Micek, Reed and Smid~\cite{JMRS}, thereby showing that, for every fixed \(s\ge3\),
\[
        f_{s,s+1}^{(2)}(n)
        =
        \Theta\bigl(\sqrt{n\log n}\bigr).
\]
The non-adjacent range \(t\ge s+2\) has a rather different flavour; see,
for example, \cite{DRR,GJ2020,JS2025} and, for generalized
Erd\H{o}s--Rogers functions, \cite{BCL,MV26,GJS}.

The graph case is therefore governed by a polynomial scale in \(n\).  For
hypergraphs the situation is qualitatively different: already in the \(3\)-uniform
case the natural scale is logarithmic, while in higher uniformities iterated
logarithms begin to appear.  The hypergraph Erd\H{o}s--Rogers function was first
considered by Dudek and Mubayi~\cite{DM}, who proved that for fixed integers
\(s\ge k\ge3\), there are constants \(c_1,c_2>0\) such that
\begin{align}\label{bound-DM}
 c_1\left(\log_{(k-2)} n\right)^{1/4}
 \le
 f_{s,s+1}^{(k)}(n)
 \le
 c_2(\log n)^{1/(k-2)}.
\end{align}
For $k=3$, they in fact obtained a sharper lower bound with an extra factor \((\frac{\log\log n}{\log\log\log n})^{1/2}\).
Here and throughout, \(\log_{(0)}x=x\), and
\(\log_{(i+1)}x=\log(\log_{(i)}x)\) denotes the iterated logarithm.
Conlon, Fox and Sudakov~\cite{CFS14}, building on their earlier work on
hypergraph Ramsey numbers~\cite{CFS10}, improved the lower bound to
\[
        f_{s,s+1}^{(k)}(n)
        \ge
        \left(\log_{(k-2)} n\right)^{1/3-o(1)}
\]
for every fixed \(s\ge k\ge3\).

The first significant barrier arises already in the \(3\)-uniform case.  Here
the upper bound of Dudek and Mubayi is logarithmic in \(n\), and no general
\(o(\log n)\) estimate had been known.  Since the available lower bounds rule
out a subpower bound in \(\log n\), the natural objective is to obtain a
nontrivial improvement over the logarithmic scale.

The situation in the \(4\)-uniform case is substantially different.
Although a bound of the form \((\log n)^{o(1)}\) is not ruled out by the
known lower bounds, prior to the present work the best general upper bound
was still a fixed positive power of \(\log n\). This led Conlon, Fox and
Sudakov~\cite[Problem~3.36]{CFS15} to pose the following problem, which was later
repeated by He and Nie~\cite[Problem~4.9]{HN}.

\begin{problem}[Conlon, Fox and Sudakov~\cite{CFS15}]\label{pbm}
For every fixed \(s\ge4\), is it the case that
\[
        f_{s,s+1}^{(4)}(n)=(\log n)^{o(1)}?
\]
\end{problem}

Problem~\ref{pbm} is closely related to a broader question arising
from the work of Dudek and Mubayi~\cite{DM}. They asked whether the
power-of-\(\log n\) upper bound in \eqref{bound-DM} could be substantially
improved, with iterated-logarithmic behaviour as the natural target.
In the diagonal adjacent case \(s=k+1\), Mubayi and Suk~\cite{MS} later
made this prediction explicit.  They proved that, for each \(k\ge14\),
\[
        f_{k+1,k+2}^{(k)}(n)
        =
        O\bigl(\log_{(k-13)}n\bigr),
\]
and conjectured that the correct order should involve \(k-2\) iterations
of the logarithm.

\begin{conjecture}[Mubayi and Suk~\cite{MS}]\label{conj:MS}
For each \(k\ge3\),
\[
        f_{k+1,k+2}^{(k)}(n)
        =
        \left(\log_{(k-2)}n\right)^{\Theta(1)}.
\]
\end{conjecture}

For \(k=4\), Conjecture~\ref{conj:MS} predicts
$
        f_{5,6}^{(4)}(n)
        =
        (\log\log n)^{\Theta(1)},
$
which is stronger than the \(s=5\) instance of
Problem~\ref{pbm}.  Thus the problem of Conlon, Fox and Sudakov may be viewed as a first
step towards the iterated-logarithmic scale predicted by the
Mubayi--Suk conjecture: it asks only for a subpower bound in \(\log n\).

Subsequently, Fan, Hu, Lin and Lu~\cite{FHLL} proved that for every fixed
\(k\ge5\),
\[
        f_{k+1,k+2}^{(k)}(n)
        =
        O\bigl(\log_{(k-3)}n\bigr).
\]
This comes within one logarithmic iteration of Conjecture~\ref{conj:MS}.  Several
recent works have studied related generalized Erd\H{o}s--Rogers functions in
\(k\)-graphs; see, for example, \cite{DHLW26S,DHLW26N,HN,DL26}.

We prove two main results.  The first establishes a new \(3\)-uniform estimate
that breaks the previously known logarithmic barrier in the adjacent problem.  The second uses
this estimate as the base case of a monotone stepping-up argument to resolve
the \(4\)-uniform problem of Conlon, Fox and Sudakov.

\begin{theorem}\label{thm:main-3}
For every fixed integer \(s\ge3\), there exists a constant \(C_s>0\) such that,
for all sufficiently large \(n\),
\[
        f_{s,s+1}^{(3)}(n)
        \le
        C_s\frac{\log n}{\log\log n}.
\]
\end{theorem}

Theorem~\ref{thm:main-3} improves the logarithmic upper bound of Dudek and
Mubayi by a factor of \(\log\log n\).  Its proof combines a robust
auxiliary \(K_s\)-free palette with an independent-set encoding of bad
local pair-colorings and a recursive hypergraph-container argument.  We
describe the underlying mechanism in more detail in the proof overview
below.

When \(s=3\), Theorem~\ref{thm:main-3} recovers, in inverse form, the following
best-known lower bound for \(r(K_4^{(3)},K_m^{(3)})\), due to Conlon, Fox and
Sudakov~\cite{CFS10}.

\begin{corollary}\label{cor:cfs-K4-lower}
For all sufficiently large \(m\),
\[
        r(K_4^{(3)},K_m^{(3)})
        \ge
        2^{\Omega(m\log m)}.
\]
\end{corollary}

\begin{proof}
Theorem~\ref{thm:main-3} with \(s=3\) gives
\(f_{3,4}^{(3)}(n)\le C\log n/\log\log n\).  Equivalently, for
\(n=\lfloor 2^{cm\log m}\rfloor\) with \(c>0\) sufficiently small, there
exists a \(K_4^{(3)}\)-free \(3\)-graph on \(n\) vertices with no independent
set of size \(m\).  Hence the claimed lower bound follows.
\end{proof}

Using a monotone stepping-up construction, we lift
Theorem~\ref{thm:main-3} to the \(4\)-uniform setting and obtain the
following subpower bound.

\begin{theorem}\label{thm:main-4}
For every fixed integer \(s\ge4\), there exists a constant \(C_s>0\) such that,
for all sufficiently large \(n\),
\[
        f_{s,s+1}^{(4)}(n)
        \le
        \exp\left(
        C_s\frac{\log\log n}{\log\log\log n}
        \right).
\]
In particular,
\[
        f_{s,s+1}^{(4)}(n)=(\log n)^{o(1)}.
\]
\end{theorem}

Theorem~\ref{thm:main-4} resolves Problem~\ref{pbm} by establishing,
for every fixed \(s\ge4\), the first general upper bound that is subpower
in \(\log n\).  This significantly improves the
\(O((\log n)^{1/2})\) bound of Dudek and Mubayi.  In the special case
\(s=4\), the inverse relation with the Ramsey number
\(r(K_5^{(4)},K_m^{(4)})\), together with recent
work~\cite{DHLW26,FLLN26}, yields the stronger estimate
$
        f_{4,5}^{(4)}(n)
        =
        O\bigl((\log\log n)^5\bigr).
$
This Ramsey-theoretic argument is specific to \(s=4\) and does not
currently provide corresponding bounds for \(s\ge5\).  In contrast,
Theorem~\ref{thm:main-4} treats the adjacent problem uniformly for every
fixed \(s\ge4\).

The \(4\)-uniform estimate has a further consequence in higher uniformities.
Combining Theorem~\ref{thm:main-4} with the stepping-up recursions of Fan, Hu,
Lin and Lu~\cite{FHLL}, we obtain the following.

\begin{corollary}\label{cor:main-k}
For every fixed integer \(k\ge5\), there exists a constant \(C_k>0\) such that,
for all sufficiently large \(n\),
\[
       f_{k+1,k+2}^{(k)}(n)
        \le
        \exp\left(
        C_k
        \frac{\log_{(k-2)} n}
             {\log_{(k-1)} n}
        \right).
\]
In particular,
$
        f_{k+1,k+2}^{(k)}(n)
        =
        \big(\log_{(k-3)} n\big)^{o(1)}.
$
\end{corollary}

Corollary~\ref{cor:main-k} gives the first upper bound of the form
\(\bigl(\log_{(k-3)}n\bigr)^{o(1)}\) for \(f_{k+1,k+2}^{(k)}(n)\), thereby
making further progress towards Conjecture~\ref{conj:MS}.  It improves the
previous \(O(\log_{(k-3)}n)\) bound of Fan, Hu, Lin and
Lu~\cite[Theorem~4.4]{FHLL}.

\medskip
The paper is organized as follows.  Section~\ref{sec:preliminaries}
collects notation and the probabilistic and container tools used later.
Section~\ref{sec:3-graphs} proves Theorem~\ref{thm:main-3}.
Section~\ref{sec:4-uniform} establishes the monotone stepping-up lemma,
deduces Theorem~\ref{thm:main-4}, and derives the higher-uniform
consequence stated in Corollary~\ref{cor:main-k}.
Finally, Section~\ref{sec:conclusion} discusses a natural strengthening
of the \(4\)-uniform result and its connection with
Conjecture~\ref{conj:MS}.

\section{Preliminaries}
\label{sec:preliminaries}

We use the following notation throughout the paper. For a hypergraph \(H\), write \(v(H)=|V(H)|\) and
\(e(H)=|E(H)|\).  We often omit braces and commas
in small sets; for instance, \(xy\) means \(\{x,y\}\), and \(xyz\) means
\(\{x,y,z\}\). For a positive integer
\(m\), write
$
        [m]=\{1,\ldots,m\}.
$
For a set \(X\) and an integer \(r\ge0\), let \(\binom{X}{r}\) denote the
family of all \(r\)-element subsets of \(X\).  We use standard asymptotic notation, and all implicit
constants may depend on the fixed parameters of the statement or proof under
consideration.

If \(H\) is a hypergraph and \(W\subseteq V(H)\), then \(H[W]\) denotes the
subhypergraph of \(H\) induced by \(W\).  For a \(k\)-graph \(H\) and an
integer \(s\ge k\), define
\[
        \alpha_s(H):=
        \max\{|W|: W\subseteq V(H),\ H[W]\text{ contains no copy of }
        K_s^{(k)}\}.
\]
Equivalently, \(\alpha_s(H)\) is the largest size of a \(K_s^{(k)}\)-free
induced subhypergraph of \(H\). Then
\[
        f_{s,t}^{(k)}(n)
        =
        \min\bigl\{
        \alpha_s(H):\ |V(H)|=n,\ K_t^{(k)}\nsubseteq H
        \bigr\}.
\]

For an \(h\)-graph \(H\) and a set \(S\subseteq V(H)\), let \(d_H(S)\)
denote the common degree of \(S\), namely
\[
        d_{H}(S)
        :=
        |\{E\in E(H): S\subseteq E\}|.
\]
For \(1\le \ell\le h\), let \(\Delta_\ell(H)\) denote the maximum \(\ell\)-codegree, that is,
\[
        \Delta_\ell(H)
        :=
        \max\left\{
        d_{H}(S):\ S\in \binom{V(H)}{\ell}
        \right\}.
\]

We use three standard tools.  The first is the following convenient
transversal-clique consequence of Janson's inequality; see, for example,
\cite[Chapter~2]{JLR} or \cite[Chapter~8]{AS}.  It gives an exponentially
small upper bound for the probability that a random graph contains no
copy of \(K_\ell\) with one vertex in each of the prescribed vertex classes.
\begin{lemma}\label{lem:janson}
For every fixed integer \(\ell\ge 2\), there exists a constant \(c_\ell>0\) with the
following property. Let \(U_1,\ldots,U_\ell\) be pairwise disjoint sets of size
\(u\), and let \(R\) be a random graph on \(U_1\cup\cdots\cup U_\ell\) in which
each edge is present independently with probability \(p\). Let \(X\) be the
number of ordered tuples \((x_1,\ldots,x_\ell)\) such that \(x_i\in U_i\) for every
\(i\) and \(x_1,\ldots,x_\ell\) form a complete graph \(K_\ell\). Then
\[
        \Pr(X=0)
        \le
        \exp\left(
        -c_\ell\min_{2\le j\le \ell}
        u^j p^{\binom{j}{2}}
        \right).
\]
\end{lemma}

\begin{proof}
For \(\ell=2\), we have
$
        \Pr(X=0)=(1-p)^{u^2}\le \exp(-u^2p).
$
Suppose \(\ell\ge3\). Let \(\mathcal A\) be the family of all ordered
transversal \(\ell\)-tuples. For \(A\in\mathcal A\), let \(E_A\) be the event
that \(A\) spans a \(K_\ell\). Then
\[
        \mu:=\sum_{A\in\mathcal A}\Pr(E_A)=u^\ell p^{\binom{\ell}{2}}.
\]
Two distinct events can be dependent only when the corresponding tuples share at
least two vertices. If they share exactly \(j\) vertices, then the number of
ordered pairs of such tuples is at most \(C_\ell u^{2\ell-j}\), and the probability
that both copies occur is \(p^{2\binom{\ell}{2}-\binom{j}{2}}\). Thus the usual
dependency term satisfies
\[
\Delta:=\sum_{A\sim B}\Pr(E_A\cap E_B)
\le
\sum_{j=2}^{\ell-1}
\binom{\ell}{j} u^{2\ell-j}\, p^{2\binom{\ell}{2}-\binom{j}{2}}
\le
C_\ell \sum_{j=2}^{\ell-1}
u^{2\ell-j}p^{2\binom{\ell}{2}-\binom{j}{2}},
\]
where \(C_\ell\) is a constant depending only on \(\ell\).
If \(\Delta=0\), then the events are mutually independent and
\(\Pr(X=0)\le \exp(-\mu)\), which gives the desired conclusion.
We may therefore assume that \(\Delta>0\).  Janson's inequality then gives
\[
        \Pr(X=0)
        \le
        \exp\left(
        -c'_\ell\min\left\{\mu,\frac{\mu^2}{\Delta}\right\}
        \right).
\]
Since \(\ell\) is fixed, we obtain
\[
        \frac{\mu^2}{\Delta}
        \ge
        c''_\ell\min_{2\le j\le \ell-1}u^j p^{\binom{j}{2}}.
\]
Combining these two estimates proves the
lemma, after adjusting the constant \(c_\ell\).
\end{proof}

The second tool is the standard supersaturation theorem for complete graphs.
It asserts that a graph whose density exceeds the Tur\'an threshold by a
fixed amount contains a positive proportion of all possible copies of the
corresponding clique.

\begin{lemma}[Supersaturation for complete graphs~\cite{ES}]\label{lem:supersaturation}
Let \(r\ge 3\) and \(\varepsilon>0\) be fixed. There are constants
\(\kappa=\kappa(r,\varepsilon)>0\) and \(t_0=t_0(r,\varepsilon)\) such that
every graph \(B\) on \(t\ge t_0\) vertices with
\[
        e(B)
        \ge
        \left(1-\frac1{r-1}+\varepsilon\right)\binom t2
\]
contains at least \(\kappa t^r\) distinct \(r\)-vertex sets spanning copies of \(K_r\).
\end{lemma}

The third is the hypergraph container method.
Hypergraph containers, developed independently by Balogh, Morris and
Samotij~\cite{BMS} and by Saxton and Thomason~\cite{ST}, have become a
powerful tool for counting independent sets in uniform hypergraphs and
describing their typical structure. They have found many applications in
extremal and Ramsey theory, as well as in additive combinatorics.  We use the
following container lemma; the stated form follows directly, for example, from \cite[Corollary~3.6]{ST}.

\begin{lemma}[Hypergraph containers~\cite{BMS,ST}]
\label{lem:containers}
For every \(h\ge2\), there are constants \(c_h,C_h,d_h,\tau_h>0\) with the
following property.  Let \(H\) be an \(h\)-graph with average degree
\(d\ge d_h\), and let \(0<\tau\le\tau_h\).  If
\[
        \Delta_\ell(H)
        \le
        c_h d\tau^{\ell-1}
        \qquad (2\le \ell\le h),
\]
then there is a family \(\mathcal C\) of subsets of \(V(H)\) such that

\medskip

(i) every independent set in \(H\) is contained in some
\(C\in\mathcal C\);

\smallskip

(ii) \(e(H[C])\le e(H)/2\) for every \(C\in\mathcal C\);

\smallskip

(iii) \(\log|\mathcal C|
\le C_h|V(H)|\tau\log(1/\tau)\).
\end{lemma}

\begin{proof}
For \(v\in V(H)\) and \(2\le\ell\le h\), let
\(\Delta_\ell(H;v)\) denote the maximum degree of an \(\ell\)-set
containing \(v\).  Thus
\[
        \Delta_\ell(H;v)\le\Delta_\ell(H)
        \qquad (v\in V(H)).
\]
In the present notation, the co-degree function of Saxton and
Thomason~\cite{ST} is
\[
        \delta(H,\tau)
        :=
        2^{\binom h2-1}
        \sum_{\ell=2}^h
        2^{-\binom{\ell-1}{2}}
        \frac{\sum_{v\in V(H)}\Delta_\ell(H;v)}
             {|V(H)|d\tau^{\ell-1}}.
\]

We use the following consequence of
\cite[Corollary~3.6]{ST}, retaining only the covering, edge-reduction,
and counting conclusions needed here. If \(0<\varepsilon,\tau<1/2\) and
\(\delta(H,\tau)\le\varepsilon/(12h!)\), then there is a family
\(\mathcal C\) of subsets of \(V(H)\) such that every independent set
of \(H\) is contained in some \(C\in\mathcal C\),
\(e(H[C])\le\varepsilon e(H)\) for every \(C\in\mathcal C\), and
\[
        \log|\mathcal C|
        \le
        A_h\log(1/\varepsilon)\,
        |V(H)|\tau\log(1/\tau),
\]
where \(A_h>0\) depends only on \(h\).

By the definition above and the bound
\(\Delta_\ell(H;v)\le\Delta_\ell(H)\), we obtain
\[
        \delta(H,\tau)
        \le
        2^{\binom h2-1}
        \sum_{\ell=2}^h
        2^{-\binom{\ell-1}{2}}
        \frac{\Delta_\ell(H)}
             {d\tau^{\ell-1}}  
        \le
        B_h
        \max_{2\le\ell\le h}
        \frac{\Delta_\ell(H)}
             {d\tau^{\ell-1}},
\]
where
$
        B_h:=
        2^{\binom h2-1}
        \sum_{\ell=2}^h
        2^{-\binom{\ell-1}{2}}.
$
The assumed maximum co-degree bounds therefore imply
\[
        \delta(H,\tau)
        \le
        B_hc_h.
\]

Take \(d_h:=1\), choose \(\tau_h>0\) sufficiently small that
Corollary~3.6 of~\cite{ST} applies whenever \(0<\tau\le\tau_h\), set
\(\varepsilon:=1/4\), and choose \(c_h>0\)
sufficiently small that \(B_hc_h\le1/(48h!)\).  Then
\[
        \delta(H,\tau)
        \le
        \frac{1}{48h!}
        =
        \frac{\varepsilon}{12h!},
\]
hence \cite[Corollary~3.6]{ST} applies.  It gives a family
\(\mathcal C\) covering every independent set of \(H\), with
\(e(H[C])\le e(H)/4\le e(H)/2\) for every \(C\in\mathcal C\), and
\[
        \log|\mathcal C|
        \le
        A_h\log4\,
        |V(H)|\tau\log(1/\tau).
\]
Taking \(C_h:=A_h\log4\) proves the lemma.
\end{proof}

\noindent
{\em Overview of the proofs of the main results.}
The proof has two main components.  For the \(3\)-uniform result, following
pair-coloring ideas of Conlon, Fox and Sudakov~\cite{CFS10} and Dudek and
Mubayi~\cite{DM}, we construct an auxiliary \(K_s\)-free graph \(F\) with a
robust transversal clique property.  We use \(F\) as a palette for randomly
coloring pairs and define a \(3\)-graph \(G_\chi\), which is
\(K_{s+1}^{(3)}\)-free by construction.

For a fixed \(t\)-vertex set, the local colorings that induce no
\(K_s^{(3)}\) are encoded as independent sets in an auxiliary hypergraph
\(\Gamma\).  A recursive application of the hypergraph container theorem,
combined with supersaturation and the robust property of \(F\), forces every
terminal container to have many reduced color lists.  Consequently, the
number of bad local colorings is at most \(q^{(1-\eta)m}\), where
\(m=\binom{t}{2}\) and \(\eta>0\).

The gain over the argument of Dudek and Mubayi occurs at this counting step.
Their fixed-palette approach passes to a packing of independent local clique
events, leading naturally to a logarithmic bound.  Our growing palette and
container argument control the full family of overlapping local
configurations.  Since the palette size grows polynomially with \(t\), the
saving \(q^{-\eta m}\) provides an additional logarithmic factor in the
exponent, which yields the factor of \(\log\log n\) improvement in
Theorem~\ref{thm:main-3}.

We next pass from \(3\)-graphs to \(4\)-graphs through a monotone
stepping-up construction.  Its edge definition ensures that a
\(K_{s+1}^{(4)}\) would induce a \(K_s^{(3)}\) in the base graph and is
therefore forbidden. Moreover, within any \(K_s^{(4)}\)-free vertex set, a sufficiently long increasing or decreasing \(\delta\)-chain would, by a standard Erdős--Szekeres type argument, determine a \(K_{s-1}^{(3)}\) in the base graph, producing a contradiction.
A Pascal-type argument converts the resulting bounds on the two chain
lengths into a bound for the whole set.  Combining this with
Theorem~\ref{thm:main-3} proves Theorem~\ref{thm:main-4}.

\section{Upper bound for \texorpdfstring{\(3\)-graphs}{3-graphs}}
\label{sec:3-graphs}

In this section we prove Theorem~\ref{thm:main-3}.  We first construct the
palette graph used in the random pair-coloring construction, and then carry
out the container count for the bad local colorings.

\subsection{The auxiliary graph \texorpdfstring{\(F\)}{F}}

The following lemma guarantees the existence of a \(K_s\)-free graph such that 
every collection of at most \(s-1\) sufficiently large vertex subsets
contains a clique with one vertex chosen from each subset.

\begin{lemma}\label{lem:robust-auxiliary}
For every fixed integer $s\ge 3$, there is a constant
$\lambda=\lambda(s)>0$ such that the following holds for every
sufficiently large integer $Q$.  There exists a $K_s$-free graph $F$ with $\frac{3Q}{4}\le |V(F)|\le Q$ such that, for every $2\le \ell\le s-1$ and every $U_1,\ldots,U_\ell\subseteq V(F)$ satisfying
\[
  |U_i|\ge |V(F)|^{1-\lambda}
  \qquad (1\le i\le \ell),
\]
there are pairwise distinct vertices $u_i\in U_i$ such that
$u_1,\ldots,u_\ell$ span a $K_\ell$ in $F$.
\end{lemma}

\begin{proof}
We choose the random graph sparse enough to contain only linearly many
copies of \(K_s\), but dense enough that every collection of at most
\(s-1\) sufficiently large vertex sets contains a transversal clique. For this purpose set
\[
  \lambda_0:=\frac1{2s},
  \qquad
  u:=\left\lceil Q^{1-\lambda_0}\right\rceil,
  \qquad
  p:=cQ^{-2/s},
\]
where \(c=c(s)>0\) will be chosen sufficiently small.  Let
\(G\sim G(Q,p)\) be an Erd\H{o}s--R\'enyi random graph on a fixed vertex
set of size \(Q\), where each possible edge is present with probability
\(p\), independently of all other edges.

Fix \(2\le \ell\le s-1\) and pairwise disjoint \(u\)-sets
\(A_1,\ldots,A_\ell\subseteq V(G)\).
By Lemma~\ref{lem:janson}, the probability that there are no vertices
\(a_i\in A_i\), \(1\le i\le \ell\), which form a \(K_\ell\), is at most 
$\exp \big(
    -c_\ell\min_{2\le r\le \ell}
    u^r p^{\binom r2}
  \big)$. For $2\le r\le s-1$,
\begin{align*}
  u^r p^{\binom r2}
  =
  Q^{r(1-1/(2s))-(2/s)\binom r2+o(1)}
  =
  Q^{r(2s-2r+1)/(2s)+o(1)}.
\end{align*}

The exponent is a concave function of $r$, and hence its minimum on
$2\le r\le s-1$ is attained at an endpoint. Note that
$\frac{2s-3}{s}\ge\frac{3(s-1)}{2s}$ ($s\ge3$),
thus we have
\[
  \min_{2\le r\le s-1}
  u^r p^{\binom r2}
  \ge
  Q^{3(s-1)/(2s)+o(1)}.
\]
The number of ordered choices of $A_1,\ldots,A_\ell$ is at most
\[
  \binom{Q}{u}^\ell
  \le
  \exp\bigl(O_s(u\log Q)\bigr)
  =
  \exp\!\left(Q^{1-1/(2s)+o(1)}\right).
\]
Since $\frac{3(s-1)}{2s}>1-\frac1{2s}$, we have, for each fixed $\ell$,
$\binom{Q}{u}^\ell
  \exp\!\left(-c_\ell Q^{3(s-1)/(2s)+o(1)}\right)=o(1)$. Taking a union bound over all ordered $\ell$-tuples
\((A_1,\ldots,A_\ell)\) of pairwise disjoint \(u\)-subsets of \(V(G)\),
and then over \(2\le \ell\le s-1\), we obtain that the event
\[
\mathcal E_1=
\left\{
\begin{array}{l}
\text{for every }2\le \ell\le s-1\text{ and every ordered tuple of pairwise}\\
\text{disjoint }u\text{-sets }A_1,\ldots,A_\ell\subseteq V(G),\text{ there is a}\\
\text{transversal copy of }K_\ell\text{ with one vertex in each }A_i
\end{array}
\right\}
\]
satisfies \(\Pr(\mathcal E_1)=1-o(1)\).

Let $X_s$ be the number of copies of $K_s$ in $G$.  Then
$ \mathbb{E}(X_s)
  =
  \binom Qs p^{\binom s2}
  \le
  C_s c^{\binom s2}Q$. Choose $c$ small enough that $\mathbb{E}(X_s)\le Q/16$.  Markov's inequality
then gives
$
 \Pr(X_s>Q/4)\le 1/4.
$
Let \(\mathcal E_2:=\{X_s\le Q/4\}\). Then, for all sufficiently large \(Q\),
\[
        \Pr(\mathcal E_1\cap\mathcal E_2)
        \ge 1-o(1)-\frac14>0.
\]
Choose a realization of \(G\) in
\(\mathcal E_1\cap\mathcal E_2\). Thus \(G\) satisfies all the
transversal properties established above and contains at most \(Q/4\)
copies of \(K_s\).
Choose one vertex from each copy of \(K_s\) and delete all chosen vertices.
The resulting induced graph $F$ is $K_s$-free and has order
$q:=|V(F)|\ge 3Q/4$.

Set \(\lambda:=1/(4s)\). Now let
\(U_1,\ldots,U_\ell\subseteq V(F)\) satisfy
$|U_i|\ge q^{1-\lambda}$, $1\le i\le \ell$.
The sets \(U_i\) are not assumed to be disjoint. Since \(q\ge 3Q/4\) and \(\ell\) is fixed,
\[
        \frac{q^{1-\lambda}}{u}=Q^{1/(4s)+o(1)}\to\infty.
\]
In particular, for all sufficiently large \(Q\), each \(U_i\) has size at least \(\ell u\). 
We may therefore choose pairwise disjoint \(u\)-sets \(A_i\subseteq U_i\), \(1\le i\le\ell\), greedily.
These sets are subsets of \(V(F)\), hence also of \(V(G)\). By the transversal
property of \(G\), there are vertices \(a_i\in A_i\), \(1\le i\le \ell\), which
form a \(K_\ell\) in \(G\). Since all these vertices lie in \(V(F)\), they also
form a \(K_\ell\) in the induced graph \(F\).
\end{proof}

\subsection{The random construction}

Fix \(s\ge3\), and set
\[
        h:=\binom{s}{2}-1,
        \qquad
        a:=\frac1{10h}.
\]
Let \(t\) be sufficiently large and set \(Q:=\lfloor t^a\rfloor\).  Apply
Lemma~\ref{lem:robust-auxiliary} with this value of \(Q\), and let \(F\) be
 the resulting \(K_s\)-free graph.  Write \(q:=|V(F)|\).  Then
\[
        q=\Theta(t^a).
\]

We now use \(F\) as a palette for a random pair-coloring.  Motivated by a
construction in~\cite{CFS10}, for an integer \(M\), choose a random map
\[
        \chi:\binom{[M]}2\longrightarrow V(F),
\]
where all values are independent and uniform.  For each outcome of this
random coloring, define the corresponding candidate \(3\)-graph \(G_\chi\)
on \([M]\) by declaring that, for \(x<y<z\),
\[
  xyz\in E(G_\chi)
  \quad\Longleftrightarrow\quad
  \chi(xy)\chi(xz)\in E(F).
\]

\begin{claim}\label{cl:no-large-clique-three}
For every choice of $\chi$, the $3$-graph $G_\chi$ is
$K_{s+1}^{(3)}$-free.
\end{claim}

\begin{proof}
Suppose that $x_1<x_2<\cdots<x_{s+1}$ span a $K_{s+1}^{(3)}$.  For every
$2\le i<j\le s+1$, the triple $x_1x_ix_j$ is an edge.  Hence
$
  \chi(x_1x_i)\chi(x_1x_j)\in E(F).
$
Since \(F\) is loopless, these colors are pairwise distinct.
It follows that
$
\chi(x_1x_2),\chi(x_1x_3),\ldots,\chi(x_1x_{s+1})
$
span a \(K_s\) in \(F\), a contradiction.
\end{proof}

\subsection{The auxiliary hypergraph \texorpdfstring{\(\Gamma\)}{Gamma}}

We first isolate the local \(t\)-vertex model for \(G_\chi\).  For a local
coloring \(\varphi:\binom{[t]}2\to V(F)\), define a \(3\)-graph \(G_\varphi\)
on \([t]\) by declaring that, for \(x<y<z\),
\[
        xyz\in E(G_\varphi)
        \quad\Longleftrightarrow\quad
        \varphi(xy)\varphi(xz)\in E(F).
\]
If \(W=\{w_1<\cdots<w_t\}\subseteq[M]\) and
\(\varphi(ij)=\chi(w_iw_j)\) for \(1\le i<j\le t\), then
\(G_\varphi\) is naturally isomorphic to \(G_\chi[W]\) via the
order-preserving map \(i\mapsto w_i\).

We now define an \(h\)-graph \(\Gamma=\Gamma_{t,F}\).  Its vertex set is
\(V(\Gamma):=\binom{[t]}2\times V(F)\).  A vertex of the form \((e,c)\),
where \(e\in\binom{[t]}2\) and \(c\in V(F)\), corresponds to assigning color
\(c\) to the pair \(e\).

Let \(v_1<\cdots<v_s\) be vertices of \([t]\).  A choice of colors
\(c_{ij}\in V(F)\), where \(1\le i<j\le s\) and \(i\le s-2\), is called
admissible if, for every \(1\le i\le s-2\), the colors
\(c_{i,i+1},c_{i,i+2},\ldots,c_{i,s}\) form a \(K_{s-i}\) in \(F\).  Every
admissible choice gives the edge
\[
        \{(v_iv_j,c_{ij}):1\le i<j\le s,\ i\le s-2\}
\]
of \(\Gamma\).  This edge has
\((s-1)+(s-2)+\cdots+2=\binom s2-1=h\) vertices, so \(\Gamma\) is an
\(h\)-graph.  (Note that the pair \(v_{s-1}v_s\) does not appear: for every
\(1\le i\le s-2\), the triple \(v_iv_{s-1}v_s\) is tested using only the two
colors assigned to \(v_iv_{s-1}\) and \(v_iv_s\).)

For a local coloring \(\varphi\), let
\[
  I_\varphi:=\left\{(xy,\varphi(xy)):xy\in\binom{[t]}2\right\}.
\]

The next lemma identifies the independence of \(I_\varphi\) in \(\Gamma\) with
the \(K_s^{(3)}\)-freeness of \(G_\varphi\).
\begin{lemma}\label{lem:pattern-equivalence}
The set \(I_\varphi\) is independent in \(\Gamma\) if and only if
\(G_\varphi\) contains no copy of \(K_s^{(3)}\).
\end{lemma}

\begin{proof}
First suppose that \(G_\varphi\) contains a copy of \(K_s^{(3)}\) on vertices
\(v_1<\cdots<v_s\).  Then, for every \(1\le i<j<k\le s\), the triple
\(v_iv_jv_k\) is an edge of \(G_\varphi\), so
\(\varphi(v_iv_j)\varphi(v_iv_k)\in E(F)\).  Define
\(c_{ij}:=\varphi(v_iv_j)\) for \(1\le i<j\le s\) and \(i\le s-2\).  Hence,
for every \(1\le i\le s-2\), the colors
\(c_{i,i+1},c_{i,i+2},\ldots,c_{i,s}\) form a complete graph in \(F\), and
the choice is admissible.  Thus
\(\{(v_iv_j,\varphi(v_iv_j)):1\le i<j\le s,\ i\le s-2\}\) is an edge of
\(\Gamma\) contained in \(I_\varphi\), so \(I_\varphi\) is not independent.

Conversely, suppose that \(I_\varphi\) is not independent in \(\Gamma\). Then
\(I_\varphi\) contains an edge. By the definition of
\(\Gamma\), this edge is supported on vertices \(v_1<\cdots<v_s\) and is determined by an admissible choice of colors \(c_{ij}\),
where $1\le i<j\le s, i\le s-2$.  Since this edge is contained in
\(I_\varphi\), we have \(c_{ij}=\varphi(v_iv_j)\) for all such pairs \(i,j\).
By admissibility, for every \(1\le i\le s-2\), the colors
\(c_{i,i+1},c_{i,i+2},\ldots,c_{i,s}\) form a complete graph in \(F\).
Thus \(\varphi(v_iv_j)\varphi(v_iv_k)\in E(F)\) for every
\(1\le i<j<k\le s\).  Hence every triple \(v_iv_jv_k\) is an edge of
\(G_\varphi\), and so \(v_1,\ldots,v_s\) form a copy of \(K_s^{(3)}\) in
\(G_\varphi\).
\end{proof}

\begin{lemma}\label{lem:pattern-codegrees}
There exists a constant \(C_s>0\) such that, for every induced subhypergraph
\(H\subseteq\Gamma\) and every \(2\le \ell\le h\), the maximum \(\ell\)-codegree of \(H\) satisfies 
\[
  \Delta_\ell(H)\le C_s t^{s-3}q^{h-\ell}.
\]
\end{lemma}

\begin{proof}
It is enough to prove the estimate for \(\Gamma\), since passing to
an induced subhypergraph can only decrease common degrees.
Fix an \(\ell\)-set
$L\subseteq V(\Gamma)$, where $2\le \ell\le h$. Each vertex of \(\Gamma\) has the form \((e,c)\), where
\(e\in\binom{[t]}2\) and \(c\in V(F)\).
First suppose that \(L\) contains two vertices of the form
$ (e,c)$ and $(e,c')$ with \(c\ne c'\).  Then no edge of \(\Gamma\) can contain \(L\),
because in one edge of \(\Gamma\) the same pair \(e\) is assigned at
most one color.  Hence the common degree of \(L\) is zero.

We may therefore assume that the vertices of \(L\) involve \(\ell\) distinct
pairs of \([t]\).  Let \(X\subseteq [t]\) be the set of all endpoints of
these pairs.  Since \(\ell\ge2\) and the pairs are distinct, we have
$ |X|\ge3$. Every edge of \(\Gamma\) containing \(L\) comes from some
\(s\)-set \(S\subseteq[t]\).  This set \(S\) must contain \(X\).  Therefore
the number of possible choices for \(S\) is at most
\[
  \binom{t}{s-|X|}
  \le C_s t^{s-|X|}
  \le C_s t^{s-3},
\]
where the estimate is interpreted as zero if \(|X|>s\).

Now fix such an \(s\)-set \(S\), and write its elements in increasing
order as $v_1<\cdots<v_s$. An edge of \(\Gamma\) supported on \(S\) is determined by assigning
colors to the \(h\) pairs $ v_iv_j, 1\le i<j\le s, i\le s-2$, subject to the admissibility condition.  The set \(L\) already fixes
\(\ell\) of these colors, provided that these pairs occur among the above
\(h\) pairs; if not, then there is no edge supported on \(S\) containing
\(L\).  In the remaining case, there are at most \(h-\ell\) colors still to
choose.  Each of them has at most \(q\) possible values, and the
admissibility condition can only reduce the number of choices.  Hence,
for this fixed \(S\), there are at most
$q^{h-\ell}$ edges of \(\Gamma\) supported on \(S\) and containing \(L\).
Combining the choices of \(S\) with the choices of the remaining colors,
we obtain
\[
  d_{\Gamma}(L)
  \le C_s t^{s-3}q^{h-\ell}.
\]
Taking the maximum over all \(\ell\)-sets \(L\subseteq V(\Gamma)\) gives
$
  \Delta_\ell(\Gamma)\le C_s t^{s-3}q^{h-\ell}.
$
The same bound holds for every induced subhypergraph
\(H\subseteq \Gamma\).  This proves the lemma.
\end{proof}

\subsection{Counting clique-free local colorings}

We now derive the key counting consequence of the robust property of \(F\):
among all \(F\)-valued local colorings, those for which \(G_\varphi\) is
\(K_s^{(3)}\)-free have an exponential saving over the trivial bound.
\begin{lemma}\label{lem:coloring-count}
There is a constant $\eta=\eta(s)>0$ such that, for all sufficiently
large $t$, the number of local colorings
$
  \varphi:\binom{[t]}2\longrightarrow V(F)
$
for which $G_\varphi$ contains no $K_s^{(3)}$ is at most
$ q^{(1-\eta)\binom t2}$.

\end{lemma}
\begin{proof}
 Set
\[
        m:=\binom t2,
        \qquad
        \rho:=\frac1{2(s-1)}.
\]
By Lemma~\ref{lem:supersaturation}, applied with \(\varepsilon=\rho\), there is a
constant \(\kappa=\kappa(s)>0\) such that every graph \(B\) on \(t\) vertices with
\(e(B)\ge(1-\rho)m\) contains at least \(\kappa t^s\) copies of \(K_s\). Set
\[
        \zeta:=\frac{\kappa}{2},
        \qquad
        \tau:=t^{-2a}.
\]
Since \(\tau\to0\) as \(t\to\infty\), we have
\(\tau\le\tau_h\) for all sufficiently large \(t\).

\medskip
\textbf{Step 1: Container applicability.}\;
We shall apply Lemma~\ref{lem:containers} recursively to induced subhypergraphs
of \(\Gamma=\Gamma_{t,F}\). Let \(C\subseteq V(\Gamma)\) satisfy
\(e(\Gamma[C])\ge \zeta t^s\), and put \(H:=\Gamma[C]\). Let
\(d(H)\) denote the average degree of \(H\). Since
\(|V(H)|\le |V(\Gamma)|=mq\), we have
\[
        d(H)
        =\frac{h e(H)}{|V(H)|}
        \ge
        c_{s,\zeta}\frac{t^{s-2}}q.
\]
As \(q=\Theta(t^a)\), this gives \(d(H)\ge c_s t^{s-2-a}\to\infty\). Thus the condition
\(d(H)\ge d_h\) in Lemma~\ref{lem:containers} holds for all sufficiently large \(t\).

By Lemma~\ref{lem:pattern-codegrees}, we have 
$
        \Delta_\ell(H)
        \le
        C_s t^{s-3}q^{h-\ell}
$ for every \(2\le \ell\le h\).
Consequently,
\[
        \frac{\Delta_\ell(H)}{d(H)\tau^{\ell-1}}
        \le
        C_s t^{-1}q^{h-\ell+1}\tau^{-(\ell-1)}
        \le
        C_s t^{-1+a(h-\ell+1)+2a(\ell-1)} = C_s t^{a(h+\ell-1)-1}.
\]
Since the exponent 
$
a(h+\ell-1)-1\le a(2h-1)-1=\frac{2h-1}{10h}-1<-\frac45,
$
we obtain that
\[
        \max_{2\le \ell\le h}
        \frac{\Delta_\ell(H)}{d(H)\tau^{\ell-1}}
        =O(t^{-4/5})=o(1).
\]
All implicit constants here depend only on the fixed parameter \(s\). Hence, for all sufficiently large \(t\), Lemma~\ref{lem:containers} applies to every induced subhypergraph \(\Gamma[C]\) with \(e(\Gamma[C])\ge \zeta t^s\).

\medskip
\textbf{Step 2: Number of terminal containers.}\;
Starting from \(V(\Gamma)\), apply Lemma~\ref{lem:containers} to every current
container \(C\) satisfying \(e(\Gamma[C])\ge \zeta t^s\), and declare \(C\)
terminal once \(e(\Gamma[C])<\zeta t^s\).  Each application to
\(\Gamma[C]\) produces a family of subcontainers \(C'\subseteq C\) such that
every independent set in \(\Gamma[C]\) is contained in one of them,
$
        e(\Gamma[C'])
        \le
        \frac12 e(\Gamma[C]),
$
and the logarithm of the number of subcontainers is at most
\(C_h|C|\tau\log(1/\tau)\).

Along any branch of this recursive procedure, the number of induced edges is
reduced by a factor of at least \(2\) at each step.  Moreover,
\[
        \frac{e(\Gamma)}{\zeta t^s}
        \le
        \frac{\binom ts q^h}{\zeta t^s}
        \le
        \frac{q^h}{\zeta}.
\]
Consequently, every branch reaches a terminal container after at most
$
        O(\log q)
$
refinement steps.

The covering property is preserved throughout the recursion: if an
independent set \(I\) of \(\Gamma\) is contained in a current container \(C\),
then \(I\) is independent in \(\Gamma[C]\), and hence after applying
Lemma~\ref{lem:containers} it is contained in one of the resulting
subcontainers.  Thus, by induction over the recursion, we obtain a family
\(\mathcal C\) of terminal containers such that every independent set of
\(\Gamma\) is contained in some \(C\in\mathcal C\), and
\[
        e(\Gamma[C])<\zeta t^s
        \qquad\text{for every }C\in\mathcal C.
\]

At each application of Lemma~\ref{lem:containers}, since \(|C|\le |V(\Gamma)|\),
the logarithm of the number of new containers is at most
$
        C_h |V(\Gamma)|\tau\log(1/\tau)
        \le
        C_h m q\tau\log t.
$
Adding this logarithmic contribution over at most \(O(\log q)\) refinement
steps, we obtain
\[
        \log|\mathcal C|
        \le
        O(mq\tau\log t\log q).
\]
Using \(q=\Theta(t^a)\) and \(\tau=t^{-2a}\), we get
$
        mq\tau\log t\log q
        =
        O\bigl(m t^{-a}(\log t)^2\bigr)
        =
        o(m\log q).
$
Thus
\[
        |\mathcal C|=q^{o(m)}.
\]

\medskip
\textbf{Step 3: Many pairs have few available colors.}\;
Fix a terminal container \(C\in\mathcal C\). For each pair \(xy\in\binom{[t]}2\), define
\[
        L_{xy}:=\{c\in V(F):(xy,c)\in C\}.
\]
Call \(xy\) large if \(|L_{xy}|\ge q^{1-\lambda}\), where \(\lambda\) is the constant from Lemma~\ref{lem:robust-auxiliary}. Let \(B_C\) be the graph on vertex set \([t]\) whose edges are precisely the large pairs.

\begin{claim}\label{Cl:lg-e}
For every terminal container \(C\in\mathcal C\),
\[
        e(B_C)<(1-\rho)m.
\]
\end{claim}

\begin{proof}
Suppose, for a contradiction, that \(e(B_C)\ge(1-\rho)m\). By the choice of \(\kappa\), the graph \(B_C\) contains at least \(\kappa t^s\) distinct \(s\)-vertex sets spanning copies of \(K_s\). Let \(v_1<\cdots<v_s\) be one such vertex set. Since every pair \(v_iv_j\) is large, for each \(1\le i\le s-2\), all lists
\[
        L_{v_iv_{i+1}}, L_{v_iv_{i+2}},\ldots,L_{v_iv_s}
\]
have size at least \(q^{1-\lambda}\). Lemma~\ref{lem:robust-auxiliary} gives vertices
\[
        c_{i,i+1}\in L_{v_iv_{i+1}},
        \quad c_{i,i+2}\in L_{v_iv_{i+2}},
        \quad \ldots,
        \quad c_{i,s}\in L_{v_iv_s}
\]
which form a \(K_{s-i}\) in \(F\).

We make these choices separately for \(i=1,\ldots,s-2\). These choices are 
compatible: the row indexed by \(i\) uses precisely the pair positions \(v_iv_j\) with
\(j>i\), and no pair position occurs in two different rows. The resulting color assignment is therefore admissible, and every chosen vertex lies in its
corresponding list.  Hence 
$
        \{(v_iv_j,c_{ij}):1\le i<j\le s,\ i\le s-2\}
$
forms an edge of \(\Gamma[C]\).

Distinct \(s\)-vertex sets spanning copies of \(K_s\) in \(B_C\) give distinct edges of \(\Gamma[C]\). Indeed, the union of the endpoints of the pair positions appearing in such a pattern edge is exactly \(\{v_1,\ldots,v_s\}\), so they determine the underlying \(s\)-set. Therefore
$
        e(\Gamma[C])\ge \kappa t^s>\zeta t^s,
$
contradicting the terminal property of \(C\). This proves the claim.
\end{proof}

\textbf{Step 4: Count bad local colorings.}\;
Fix a terminal container \(C\in\mathcal C\).  Recall that
\[
        I_\varphi=\left\{(xy,\varphi(xy)):xy\in\binom{[t]}2\right\},
        \qquad
        L_{xy}=\{c\in V(F):(xy,c)\in C\}.
\]
For each pair \(xy\in\binom{[t]}2\), the set \(I_\varphi\) contains the
single vertex \((xy,\varphi(xy))\) of \(V(\Gamma)\) lying over \(xy\).
Therefore \(I_\varphi\subseteq C\) if and only if
\(\varphi(xy)\in L_{xy}\) for every \(xy\in\binom{[t]}2\).  Hence the number
of local colorings \(\varphi\) with \(I_\varphi\subseteq C\) is at most
$
        \prod_{xy\in\binom{[t]}2}|L_{xy}|.
$

By Claim~\ref{Cl:lg-e}, at least \(\rho m\) pairs \(xy\in\binom{[t]}2\) satisfy
$
        |L_{xy}|\le q^{1-\lambda}.
$
For all remaining pairs, we use the trivial bound \(|L_{xy}|\le q\).  Therefore
\[
        \prod_{xy\in\binom{[t]}2}|L_{xy}|
        \le
        \bigl(q^{1-\lambda}\bigr)^{\rho m} q^{(1-\rho)m}
        =
        q^{(1-\lambda\rho)m}.
\]
Thus, for each fixed terminal container \(C\), the number of local colorings
\(\varphi\) with \(I_\varphi\subseteq C\) is at most
\(q^{(1-\lambda\rho)m}\).
By Lemma~\ref{lem:pattern-equivalence}, every local coloring \(\varphi\) for which \(G_\varphi\) contains no \(K_s^{(3)}\) gives an independent set \(I_\varphi\) in \(\Gamma\). Since every independent set is contained in some terminal container, the total number of such local colorings is at most
\[
        |\mathcal C|q^{(1-\lambda\rho)m}
        =
        q^{o(m)}q^{(1-\lambda\rho)m}.
\]
Take \(\eta:=\lambda\rho/2\). For all sufficiently large \(t\), the factor \(q^{o(m)}\) is at most \(q^{\eta m}\). Hence the total number of local colorings \(\varphi\) for which \(G_\varphi\) contains no copy of \(K_s^{(3)}\) is at most
\[
        q^{(1-\lambda\rho+\eta)m}
        =
        q^{(1-\eta)m}
        =
        q^{(1-\eta)\binom t2}.
\]
This proves the lemma.
\end{proof}

\begin{proof}[Proof of Theorem~\ref{thm:main-3}]
Let \(\eta\) be given by Lemma~\ref{lem:coloring-count}, and set
\[
        M:=\left\lfloor q^{\eta(t-1)/4}\right\rfloor.
\]
By Claim~\ref{cl:no-large-clique-three}, every \(3\)-graph \(G_\chi\) in the
construction is \(K_{s+1}^{(3)}\)-free.  Fix
\(W=\{w_1<\cdots<w_t\}\subseteq[M]\), and define
\(\varphi_W:\binom{[t]}2\to V(F)\) by
\(\varphi_W(ij)=\chi(w_iw_j)\) for \(1\le i<j\le t\).  Then
\(G_{\varphi_W}\) is naturally isomorphic to \(G_\chi[W]\) via the
order-preserving map \(i\mapsto w_i\).  Since \(\varphi_W\) is uniformly
distributed over all local colorings \(\binom{[t]}2\to V(F)\),
Lemma~\ref{lem:coloring-count} gives
\[
        \Pr\bigl(K_s^{(3)}\nsubseteq G_\chi[W]\bigr)
        \le q^{-\eta\binom t2}.
\]
Consequently, the expected number of \(K_s^{(3)}\)-free \(t\)-sets is at most
\[
        \binom Mt q^{-\eta\binom t2}
        \le
        M^tq^{-\eta t(t-1)/2}
        \le
        q^{\eta t(t-1)/4-\eta t(t-1)/2}<1.
\]
Thus there exists a coloring \(\chi\) for which \(G_\chi\) is
\(K_{s+1}^{(3)}\)-free and every \(t\)-set contains a copy of \(K_s^{(3)}\).
Equivalently, \(\alpha_s(G_\chi)<t\), and hence
\(f^{(3)}_{s,s+1}(M)<t\).

Since \(q=\Theta(t^a)\), we have
$
        \log M=\Theta(t\log t).
$
Given a sufficiently large \(n\), choose
\[
        t=\left\lceil C_s\frac{\log n}{\log\log n}\right\rceil
\]
with \(C_s\) large enough. Then \(\log M\ge \log n\) for all sufficiently
large \(n\), and hence the corresponding value of \(M\) is at least \(n\).
Take an induced subhypergraph \(G'\) of \(G_\chi\) on any \(n\) vertices. It
remains \(K_{s+1}^{(3)}\)-free, and every \(K_s^{(3)}\)-free set in \(G'\) is
also a \(K_s^{(3)}\)-free set in \(G_\chi\). Hence \(\alpha_s(G')<t\), and
therefore after enlarging \(C_s\) if necessary, this yields
\[
        f^{(3)}_{s,s+1}(n)
        \le
        C_s\frac{\log n}{\log\log n}.
\]
This completes the proof of Theorem~\ref{thm:main-3}.
\end{proof}

\section{The \texorpdfstring{\(4\)-uniform}{4-uniform} bound}
\label{sec:4-uniform}

We now prove Theorem~\ref{thm:main-4} using a monotone stepping-up
construction from \(3\)-graphs to \(4\)-graphs.  Starting with a \(3\)-graph on
\([N]\), we pass to the ordered binary cube \(\{0,1\}^N\) and define edges using
the consecutive \(\delta\)-values.  The monotonicity condition preserves the
forbidden clique condition and allows large \(K_s^{(4)}\)-free sets to be
controlled by increasing and decreasing \(\delta\)-chains.  Theorem~\ref{thm:main-3}
then supplies the required bound for the base \(3\)-graph.

\subsection{Binary sequences and monotone chains}

Let $\Omega_N:=\{0,1\}^N$. We order $\Omega_N$ by the value
$
  \operatorname{val}(x):=\sum_{i=1}^Nx_i2^{i-1}.
$
For $x<y$, define
\[
  \delta(x,y):=\max\{i:x_i\ne y_i\}.
\]
Throughout this section, whenever an ordered sequence
\(x_1<\cdots<x_m\) in \(\Omega_N\) is under consideration, we write
\[
        d_i:=\delta(x_i,x_{i+1}),\qquad 1\le i<m.
\]

The following elementary properties of the \(\delta\)-function are
standard in stepping-up arguments and follow directly from its definition;
see, for example, \cite{GRS} for the classical stepping-up framework.

\begin{lemma}\label{lem:delta-properties}
Let $m\ge2$, and let $x_1<\cdots<x_m$ be elements of $\Omega_N$.

\medskip
  (i) If $x<y<z$, then
 $
    \delta(x,z)=\max\{\delta(x,y),\delta(y,z)\}.
 $
  
  \medskip
  (ii) For $1\le i<j\le m$,
  $
    \delta(x_i,x_j)=\max\{d_i,d_{i+1},\ldots,d_{j-1}\}.
  $
  
   \medskip
  (iii) The maximum of $d_1,\ldots,d_{m-1}$ is attained at a unique index.

\end{lemma}

\begin{proof}

Part (i) follows by looking at the largest coordinate at which one of the
three strings changes.  Part (ii) follows from repeated use of part (i).
For part (iii), suppose that the maximum value \(d\) is attained at
indices \(i<k\).  By part (ii),
\(\delta(x_i,x_k)=\max\{d_i,\ldots,d_{k-1}\}=d\), so
\((x_k)_d=1\).  But \(d_k=d\) and \(x_k<x_{k+1}\) imply
\((x_k)_d=0\), a contradiction.
\end{proof}

The following elementary lemma is a Pascal-type form of the classical
Erd\H{o}s--Szekeres monotone subsequence argument, adapted to the
\(\delta\)-chains used in the stepping-up construction.  

For every nonempty \(A\subseteq\Omega_N\), let \(I(A)\) be the largest \(p\) for which there are
\(y_0<y_1<\cdots<y_p\) in \(A\) with
\[
        \delta(y_0,y_1)<\delta(y_1,y_2)<\cdots<
        \delta(y_{p-1},y_p).
\]
Thus \(I(A)\) is the maximum number of consecutive \(\delta\)-steps in an
ordered sequence from \(A\) whose \(\delta\)-values are strictly increasing.
Define $D(A)$ in the same way with all inequalities reversed.

\begin{lemma}\label{lem:pascal-chain}
For every nonempty \(A\subseteq \Omega_N\),
\[
        |A|\le \binom{I(A)+D(A)}{I(A)}.
\]
\end{lemma}

\begin{proof}
We use induction on \(|A|\). The result is clear when \(|A|=1\). Write
\(A=\{x_1<\cdots<x_m\}\), and let the unique largest value among
\(d_1,\ldots,d_{m-1}\) occur at \(d_j\). Put
\[
        A_L:=\{x_1,\ldots,x_j\},\qquad
        A_R:=\{x_{j+1},\ldots,x_m\}.
\]

For every \(z\in A_L\), Lemma~\ref{lem:delta-properties} gives \(\delta(z,x_{j+1})=d_j\). All
\(\delta\)-values occurring inside \(A_L\) are smaller than \(d_j\). Hence every increasing
\(\delta\)-chain in \(A_L\) can be extended by appending \(x_{j+1}\), and therefore
\(I(A_L)\le I(A)-1\). Symmetrically, every decreasing \(\delta\)-chain in \(A_R\)
can be extended to the left by \(x_j\), and hence \(D(A_R)\le D(A)-1\). Also,
\(D(A_L)\le D(A)\) and \(I(A_R)\le I(A)\).

By the induction hypothesis and Pascal's identity,
\[
\begin{aligned}
|A|
=|A_L|+|A_R|  
&\le
\binom{I(A_L)+D(A_L)}{I(A_L)}
+
\binom{I(A_R)+D(A_R)}{I(A_R)} \\
&\le
\binom{I(A)+D(A)-1}{I(A)-1}
+
\binom{I(A)+D(A)-1}{I(A)} \\
&=
\binom{I(A)+D(A)}{I(A)}.
\end{aligned}
\]
This proves the lemma.
\end{proof}

\subsection{The stepping-up construction}

\begin{lemma}\label{lem:stepping-up}
Let \(s\ge4\) and \(N\ge1\).  Put
$
        \rho:=f^{(3)}_{s-1,s}(N).
$
Then
\[
        f^{(4)}_{s,s+1}(2^N)
        \le
        \binom{2\rho}{\rho}
        \le 4^\rho.
\]
\end{lemma}

\begin{proof}
By the definition of $\rho$, there exists an $N$-vertex $3$-graph $G$ such that $K_s^{(3)}\nsubseteq G$, while every subset of more than $\rho$ vertices contains a copy of $K_{s-1}^{(3)}$. We now construct a $4$-graph $H$ on the vertex set $\Omega_N=\{0,1\}^N$. For four vertices $x_1<x_2<x_3<x_4$ in the usual order, we declare $x_1x_2x_3x_4$ to be an edge of $H$ if and only if the three consecutive $\delta$-values $d_1,d_2,d_3$ (where $d_i=\delta(x_i,x_{i+1})$) are either strictly increasing or strictly decreasing, and moreover $\{d_1,d_2,d_3\}\in E(G)$.

We first verify that $H$ contains no $K_{s+1}^{(4)}$. Assume for contradiction that $x_1<\cdots<x_{s+1}$ span a $K_{s+1}^{(4)}$ in $H$. Then every three consecutive terms $d_i,d_{i+1},d_{i+2}$ must be monotone (either all increasing or all decreasing). If two consecutive triples had opposite directions, then the shared values $d_{i+1},d_{i+2}$ would have to satisfy both $d_{i+1}<d_{i+2}$ and $d_{i+1}>d_{i+2}$, which is impossible. Hence all triples have the same orientation, so we have either
\[
d_1<d_2<\cdots<d_s
\qquad
\text{or}
\qquad
d_1>d_2>\cdots>d_s.
\]
Consider first the increasing case. For any $1\le i<j<k\le s$, look at the four vertices $x_i,x_{i+1},x_{j+1},x_{k+1}$. By Lemma~\ref{lem:delta-properties}, their consecutive $\delta$-values are exactly $d_i,d_j,d_k$. Since these four vertices form an edge of $H$, we must have $\{d_i,d_j,d_k\}\in E(G)$. Therefore the set $\{d_1,\ldots,d_s\}$ spans a copy of $K_s^{(3)}$ in $G$, contradicting the choice of $G$. The decreasing case is symmetric: we use the four vertices $x_i,x_j,x_k,x_{k+1}$ instead, and the same argument yields a contradiction. Thus $H$ is indeed $K_{s+1}^{(4)}$-free.

Now let $\emptyset\ne A\subseteq V(H)$ be such that $H[A]$ contains no copy of $K_s^{(4)}$. We claim that both $I(A)\le \rho$ and $D(A)\le \rho$. 

To prove $I(A)\le \rho$, take an increasing $\delta$-chain $y_0<y_1<\cdots<y_p$ inside $A$, and set $e_i:=\delta(y_{i-1},y_i)$ for $1\le i\le p$. Then we have
$
e_1<e_2<\cdots<e_p.
$
We shall show that $\{e_1,\dots,e_p\}$ cannot contain a copy of $K_{s-1}^{(3)}$ in $G$. Suppose otherwise, and let such a copy be indexed by $e_{i_1},e_{i_2},\ldots,e_{i_{s-1}}$ with $i_1<\cdots<i_{s-1}$. Consider the $s$ vertices
\[
y_{i_1-1},\; y_{i_1},\; y_{i_2},\; \ldots,\; y_{i_{s-1}}.
\]
Because the \(e_i\)'s are strictly increasing, for each pair of
consecutive selected vertices, the maximum of the intervening
\(e\)-values is the last one. Hence the consecutive $\delta$-values are exactly
$
e_{i_1}, e_{i_2}, \ldots, e_{i_{s-1}}.
$
Now take any four of these selected vertices. Their three consecutive $\delta$-values are of the form $e_{i_\alpha}, e_{i_\beta}, e_{i_\gamma}$ with $1\le \alpha<\beta<\gamma\le s-1$, and they are strictly increasing. Since the set $\{e_{i_1},\dots,e_{i_{s-1}}\}$ spans a $K_{s-1}^{(3)}$ in $G$, the triple $\{e_{i_\alpha},e_{i_\beta},e_{i_\gamma}\}$ is an edge of $G$. By the definition of $H$, every four of the chosen vertices form an edge of $H$. Thus these $s$ vertices span a copy of $K_s^{(4)}$ in $H$, contradicting the assumption on $A$. Therefore \(\{e_1,\ldots,e_p\}\) is \(K_{s-1}^{(3)}\)-free in \(G\).
Since the \(e_i\)'s are pairwise distinct, the choice of \(G\) implies that
\(p\le\rho\). Hence \(I(A)\le\rho\).

The proof of $D(A)\le \rho$ is analogous. Suppose we have a decreasing chain with $e_1>e_2>\cdots>e_p$, and assume that $e_{i_1},\ldots,e_{i_{s-1}}$ span a $K_{s-1}^{(3)}$ in $G$. This time select the $s$ vertices
\[
y_{i_1-1},\; y_{i_2-1},\; \ldots,\; y_{i_{s-1}-1},\; y_{i_{s-1}}.
\]
Since the $e_i$'s are strictly decreasing, for each pair of consecutive selected vertices, the maximum of the
intervening \(e\)-values is the first one. Hence the consecutive $\delta$-values are again
$
e_{i_1}, e_{i_2}, \ldots, e_{i_{s-1}},
$
but now in strictly decreasing order. Consequently, any four of these vertices have three consecutive $\delta$-values $e_{i_\alpha},e_{i_\beta},e_{i_\gamma}$ with $\alpha<\beta<\gamma$, which are strictly decreasing. Since the corresponding triple is an edge of $G$, every four of the chosen vertices form an edge of $H$, yielding a copy of $K_s^{(4)}$ in $H$, again a contradiction. Hence $D(A)\le \rho$.

Finally, applying Lemma~\ref{lem:pascal-chain} yields
\[
|A| \le \binom{I(A)+D(A)}{I(A)} \le \binom{2\rho}{\rho} \le 4^\rho.
\]
Since this holds for every $K_s^{(4)}$-free set $A\subseteq V(H)$, the lemma is proved.
\end{proof}

\begin{proof}[Proof of Theorem~\ref{thm:main-4}]
Fix $s\ge4$, and set $N:=\lceil\log_2 n\rceil$. By Theorem~\ref{thm:main-3}, we have
\[
\rho:=f^{(3)}_{s-1,s}(N) \le C_s\frac{\log N}{\log\log N}.
\]
Lemma~\ref{lem:stepping-up} then provides a $K_{s+1}^{(4)}$-free $4$-graph $H$ on $2^N$ vertices in which every $K_s^{(4)}$-free set has size at most $4^\rho$. Since $n\le 2^N$, we may take an induced subhypergraph $H'$ of $H$ on $n$ vertices. It remains $K_{s+1}^{(4)}$-free, and any $K_s^{(4)}$-free set in $H'$ is also $K_s^{(4)}$-free in $H$. Therefore
\[
f^{(4)}_{s,s+1}(n) \le 4^\rho \le \exp\!\left(C_s'\frac{\log N}{\log\log N}\right).
\]
As $N=\Theta(\log n)$, we have
\[
\frac{\log N}{\log\log N} = O\!\left(\frac{\log\log n}{\log\log\log n}\right),
\]
so
\[
f^{(4)}_{s,s+1}(n) \le \exp\!\left(C_s''\frac{\log\log n}{\log\log\log n}\right).
\]
Finally,
\[
\exp\!\left(C_s''\frac{\log\log n}{\log\log\log n}\right)
= (\log n)^{C_s''/\log\log\log n}
= (\log n)^{o(1)}.
\]
This completes the proof of Theorem~\ref{thm:main-4}.
\end{proof}

\subsection{Higher-uniform consequences}
\label{subsec:higher}

We conclude by deriving Corollary~\ref{cor:main-k}.  The starting point is
Theorem~\ref{thm:main-4} in the case \(s=5\), which we combine with the
following stepping-up recursions of Fan, Hu, Lin and Lu~\cite[Theorems~4.5
and~4.6]{FHLL}.

\begin{lemma}[Fan, Hu, Lin and Lu~\cite{FHLL}]
\label{lem:FHLL-step}
There exists a constant \(C>0\), and for every \(k\ge5\) a constant
\(C_k>0\), such that, for every \(N\ge1\),
\[
        f_{6,7}^{(5)}(2^N)
        \le
        C\bigl(f_{5,6}^{(4)}(N)\bigr)^2,
        \qquad
        f_{k+2,k+3}^{(k+1)}(2^N)
        \le
        C_k f_{k+1,k+2}^{(k)}(N).
\]
\end{lemma}

\begin{proof}[Proof of Corollary~\ref{cor:main-k}]
Fix an integer \(k\ge5\).  Put
\[
        n_0:=n,
        \qquad
        n_i:=\left\lceil \log_2 n_{i-1}\right\rceil
        \quad (1\le i\le k-4).
\]
Then, for each fixed \(i\),
\[
        n_i=\Theta(\log_{(i)}n),
        \qquad
        n_{i-1}\le 2^{n_i}.
\]
We shall use the following simple monotonicity observation.  If a construction
on \(2^{n_i}\) vertices gives an upper bound for a function \(f\), then by
restricting to an induced subhypergraph on \(n_{i-1}\) vertices, the same upper
bound holds with \(2^{n_i}\) replaced by \(n_{i-1}\).  Indeed, the forbidden
clique condition is preserved under induced subhypergraphs, and the largest
\(K_s\)-free set cannot increase.

Iterating the second inequality in Lemma~\ref{lem:FHLL-step} \(k-5\) times
(with no iteration when \(k=5\)), and then applying the first inequality once,
we get
\[
        f_{k+1,k+2}^{(k)}(n)
        \le
        C_k\left(f_{5,6}^{(4)}(n_{k-4})\right)^2.
\]
Indeed, each linear step lowers the uniformity and the two clique parameters
by one, while the final \(4\)-to-\(5\) step contributes the square.

By Theorem~\ref{thm:main-4} with \(s=5\),
\[
        f_{5,6}^{(4)}(n_{k-4})
        \le
        \exp\left(
        C
        \frac{\log\log n_{k-4}}
             {\log\log\log n_{k-4}}
        \right).
\]
Since \(n_{k-4}=\Theta(\log_{(k-4)}n)\), the last exponent is
\(O_k(\log_{(k-2)}n/\log_{(k-1)}n)\).  Absorbing the square and all fixed
constants into \(C_k\), we obtain
\[
        f_{k+1,k+2}^{(k)}(n)
        \le
        \exp\left(
        C_k
        \frac{\log_{(k-2)} n}
             {\log_{(k-1)} n}
        \right)
        =
        \left(\log_{(k-3)}n\right)^{o(1)}.
\]
This completes the proof.
\end{proof}

\section{Concluding remarks}
\label{sec:conclusion}

We have proved that \(f_{s,s+1}^{(4)}(n)=(\log n)^{o(1)}\) for every fixed
\(s\ge4\), resolving Problem~\ref{pbm} of Conlon, Fox and Sudakov.  Together
with the stepping-up recursions of Fan, Hu, Lin and Lu~\cite{FHLL}, this also
yields the higher-uniform consequence in Corollary~\ref{cor:main-k}.

Recall that the general lower bound of Conlon, Fox and
Sudakov~\cite{CFS14}, in particular when \(k=4\), gives
\[
        f_{s,s+1}^{(4)}(n)
        \ge
        (\log\log n)^{1/3-o(1)}
\]
for every fixed \(s\ge4\).  Thus the known lower bound already lies on a
polynomial scale in \(\log\log n\).  For \(s=4\), a polynomial upper bound
in \(\log\log n\) follows from the recent Ramsey lower bounds discussed in
the introduction.  This leaves the following natural question for every
fixed \(s\ge5\).

\begin{problem}\label{pbm-LN}
For every fixed \(s\ge5\), does there exist a constant \(C_s>0\) such that
\[
        f_{s,s+1}^{(4)}(n)
        \le
        (\log\log n)^{C_s}
\]
for all sufficiently large \(n\)?
\end{problem}

The present monotone stepping-up argument incurs an exponential loss in
the \(3\)-uniform parameter and therefore does not yield a polynomial
bound in \(\log\log n\).  Reaching this scale appears to require a more
efficient lifting mechanism or a direct \(4\)-uniform construction.

As discussed in the introduction, an affirmative answer to
Problem~\ref{pbm-LN} for \(s=5\) would, together with
Lemma~\ref{lem:FHLL-step}, settle Conjecture~\ref{conj:MS} in full.

\end{document}